\documentclass{amsart}
\usepackage{indentfirst,latexsym,bm,amsmath,amssymb,amsthm}
\usepackage{enumerate}
\newtheorem{theorem}{Theorem}[section]
\newtheorem{Lemma}{Lemma}[section]
\newtheorem{proposition}{Proposition}[section]

\newtheorem{remark}{Remark}[section]

\numberwithin{equation}{section}

\begin{document}

\title{Global Large Smooth Solutions for 3-D Hall-magnetohydrodynamics}

\author{Huali Zhang}
\address{School of Mathematics and Statistics, Changsha University
of Science and Technology, Changsha 410114, People's Republic of China.}
\email{zhlmath@yahoo.com}



\date{\today}



\begin{abstract}
In this paper, the global smooth solution of Cauchy's problem of incompressible, resistive, viscous Hall-magnetohydrodynamics (Hall-MHD) is studied. By exploring the nonlinear structure of Hall-MHD equations, a class of large initial data is constructed, which can be arbitrarily large in $H^3(\mathbb{R}^3)$. Our result may also be considered as the extension of work of Lei-Lin-Zhou \cite{LLZ} from the second-order semilinear equations to the second-order quasilinear equations, because the Hall term elevates the Hall-MHD system to the quasilinear level.
\end{abstract}

\maketitle



\section{Introduction and Main results}
In this paper we consider the following incompressible, resistive, viscous Hall-MHD equations
\begin{equation}\label{HMHD}
\begin{cases}
u_t+u \cdot \nabla u+\nabla p= \nu \Delta u+b \cdot \nabla b,\\
b_t+u \cdot \nabla b-b \cdot \nabla u+\eta \nabla \times \left( \left(\nabla \times b \right) \times b\right)=\mu \Delta b,\\
\nabla\cdot u=0, \quad \nabla\cdot b=0,\\
u|_{t=0}=u_0, \ b|_{t=0}=b_0,
\end{cases}
\end{equation}
on the domain $(t,x) \in \mathbb{ R}^+ \times \mathbb{R}^3$,
where $u=(u_1,u_2,u_3)^{\text{T}}, b=(b_1,b_2,b_3)^{\text{T}}\in \mathbb{R}^3$ denote the fluid velocity and magnetic fields respectively. The scalars $p, \nu, \mu, \eta$ are the pressure, viscosity, magnetic diffusivity, Hall effect coefficient respectively ($\nu$, $\mu$ and $\eta$ are positive constants). $u_0$ and $b_0$ are the initial data satisfying
\begin{equation}\label{HMHD1}
\nabla \cdot u_0=\nabla \cdot h_0=0.
\end{equation}

Equation (1.1) is important to describe some physical phenomena, e.g., space plasmas, star formation, neutron stars and dynamo, see for \cite{L,BT,F,HG,MGM,SU,W} and references therein. In the case $\eta=0$, the Hall-MHD equation reduces to the standard MHD equation which has been extensively researched, and there exists a lot of excellent works, see instances, \cite{ST,CM,CG,Lin, LLZ,LZZ,RWX,ZZ}.

For Hall-MHD equations, some newly developments have been made. For example, Chae et al. in \cite{CDL} proved global smooth solutions of three-dimensional Hall-MHD equation with small initial data in $\left( H^3(\mathbb{R}^3) \right)^3 \times \left( H^3(\mathbb{R}^3) \right)^3$. Chae and Lee improved their results under weaker smallness assumptions on the initial data, see \cite{CL} for details. There are other prominent works for small solutions for Hall-MHD equations, for examples, \cite{AD,D,Zhang,YM,CWW,CW,CW2}. These results of the global well-posedness of the three-dimensional Hall-MHD system under the smallness condition on the initial data in the deterministic case requires positive diffusion on both velocity and magnetic fields equations. However, with noise, zero viscosity is allowed;
Yamazaki and Moha in \cite{YM} proved the global well-posedness of the three- dimensional stochastic Hall-MHD system with zero viscosity under the smallness condition on the initial data.
%
However, none of results are known for Hall-MHD equations for general initial data without smallness conditions. Under a class of large initial data, we found some results for incompressible Navier-Stokes equations and incompressible standard MHD equations, see \cite{CM, CG, LLZ, LZZ, ZZ} for details. 
Those motivate us to study the global well-posedness of Cauchy's problem of Hall-MHD equations with large inital data. But 
the Hall term heightens the level of nonlinearity of the standard MHD system from a second-order semilinear to a second-order quasilinear level, significantly making its qualitative analysis more difficult. To the author's knowledge, it's quite rare to prove the existence of large, smooth, global solutions for quasilnear system. Using the large initial data constructed in \cite{LLZ}, the author is very fortunately to go through these difficulties by combining nonlinear structures and commutator energy estimates for resistive, viscous Hall-MHD equations.

The aim of this paper is to prove the existence of a unique, global smooth solution of Hall-MHD equations with initial data being arbitrarily large in $H^3(\mathbb{R}^3)$. Our result completely drops the smallness condition on the initial data.

Before we state our main results, we first give some notations. Let $\chi_{M_0}(x):=\chi(\frac{x}{M_0})$, $M_0$ is a positive constant, $\chi(x) \in C^{\infty}_0(\mathbb{R}^3)$ is a cut off function satisfying $|\chi(x)| \leqslant 2$ and
\begin{align}\label{102}
&\chi(x)\equiv1, \quad \text{for} \ |x| \leqslant 1; \quad \chi(x)\equiv 0, \quad \text{for} \ |x| \geqslant 2,
\\
& |\nabla^k \chi(x)| \leqslant 2, \quad 0 \leqslant k \leqslant 5.
\end{align}
Let $v_0$ be that constructed by Lei et al. \cite{LLZ}, and it has the following properties
\begin{align}\label{103}
&\nabla \cdot v_0=0, \quad \nabla \times v_0=\sqrt{-\Delta}v_0,
\\
& \text{supp} \hat{v}_0 \subseteq \{ \xi| 1-\delta \leqslant |\xi| \leqslant 1+\delta \}, \quad 0<\delta\ \leqslant \frac{1}{2},
\\
&||\hat{v}_0||_{L^1} \leqslant M_1, \quad |\nabla^k v_0| \leqslant \frac{M_2}{1+|x|}, \quad 0\leqslant k \leqslant 5,
\end{align}
where $M_1, M_2$ are positive constants, $\hat{v}_0$ is the Fourier transform of $v_0$ and the operator $\sqrt{-\Delta}$ is defined through the Fourier transform
\begin{equation*}
\widehat{\sqrt{-\Delta}f}(\xi)=|\xi|\hat{f}(\xi).
\end{equation*}

Our main result is as follows.
\begin{theorem}\label{thm}
	Consider Cauchy's problem (1.1)-(1.2). Suppose that  	
\begin{align}\label{104}
	&\quad u_0=u_{01}+\chi_{M_0} u_{02},
	\\
	&\quad b_0=b_{01}+\chi_{M_0} b_{02}.
\end{align}	
	with
\begin{align}\label{100}
&\nabla \cdot u_{02}=\nabla \cdot b_{02}=0,
\\
&u_{02}=\alpha_1 v_0, \quad b_{02}=\alpha_2 v_0,
\end{align}
where $\chi_{M_0},v_0$ are stated as above. $\alpha_1,\alpha_2$ are two real constants. Then there exist constants $\delta^{-\frac{1}{2}} \geqslant M_0 \gg 1$ depending on $M_1, M_2, \alpha_1, $ $\alpha_2, \mu, \nu, \eta$ such that Equations (1.1)-(1.2) has a unique, global smooth solution provided that
\begin{equation}\label{111}
||u_{01}||_{H^3}+||b_{01}||_{H^3} \leqslant M_0^{-\frac{1}{2}}.
\end{equation}
\end{theorem}
\begin{remark}
For  
\begin{equation*}
||u_0||_{L^\infty}+||b_0||_{L^\infty}\lesssim M_1, 
\end{equation*}
\begin{equation*}
||u_0||_{H^3}+||b_0||_{H^3}\lesssim M_0^{-\frac{1}{2}}+(|\alpha_1|+|\alpha_2|)\sum_{k=0}^{3}\frac{M_2}{M_0^k},
\end{equation*}
and the constant $M_1,M_2$ can be arbitrary large, thus our initial data can be arbitrary large.
\end{remark}
\begin{remark}
In the limiting case $\delta =
 0$, $\nabla \times u_{02} = u_{02}$, $\nabla \times b_{02} = b_{02}$, and the flow, magnetic field are called Beltrami flow and force-free fields respectively. Let us also mention that the magnetic energy achieves the minimum value for force-free fields, one can refer \cite{T} for details.
\end{remark}
The proof of Theorem \ref{thm} is based on a perturbation argument along with a standard cut-off technique, and the perturbation is as large as the initial data. Compared with the standard MHD equations, a part of the nonlinearities may not be small for Hall-MHD equations (see \eqref{B} and \eqref{000}). Fortunately, by combining the nonlinear structure of the term and commutator estimates, these terms can be estimated carefully.

This paper is organized as follows: In section 2, we introduce commutator estimates and give some estimates of some quadratic terms. Section 3 is devoted to prove the global existence and uniqueness of large smooth solutions for Hall-MHD equations.
\section{Preliminaries}
In this section, we first give some notations. Let $\alpha=(\alpha_1,\alpha_2,\alpha_3)$ be a multi index, $|\alpha|=\sum^3_{i=1} \alpha_i$, $\partial=(\partial_{x_1}, \partial_{x_2}, \partial_{x_3})$ and $\partial^\alpha=(\partial^{\alpha_1}_{x_1}, \partial^{\alpha_2}_{x_2}, \partial^{\alpha_3}_{x_3})$.

Let $m$ be a positive integer and $r>0$.  We use the notation $||g||_{H^m(|x| \leqslant r)}$ to denote the Sobolev norm localized in bounded domain $\left\{ x \in \mathbb{R}^3 \big| |x| \leqslant r \right\}$, that is,
\begin{equation*}
||g||_{H^m(|x| \leqslant r)}:=\sum_{0 \leqslant |\alpha| \leqslant m} \left( \int_{|x|\leqslant r} |\partial^\alpha g|^2 dx\right)^{\frac{1}{2}}.
\end{equation*}
Next, we introduce the commutator estimate.
\begin{Lemma}\label{lemc}\cite{CDL}
	Let $m$ be a positive integer, $h,v \in H^m(\mathbb{R}^3)$. The following commutator estimate
\begin{equation}\label{200}
\sum_{|\alpha| \leqslant m} ||D^\alpha(hv)-(D^\alpha h)v||_{L^2} \leqslant C\left( ||h||_{H^{m-1}}||\nabla v||_{L^\infty}+ ||h||_{L^\infty}||v||_{H^m}\right)
\end{equation}
holds.
\end{Lemma}
Let $f,g$ satisfy
\begin{equation}\label{300}
\begin{cases}
f_t-\nu \Delta f=0,\\
t=0: f=u_{02},
\end{cases}
\end{equation}
and
\begin{equation}\label{301}
\begin{cases}
g_t-\mu \Delta g=0,\\
t=0: g=b_{02}.
\end{cases}
\end{equation}
Therefore, we have
\begin{equation*}
f=e^{\nu t\Delta}u_{02}, \quad g=e^{\mu t\Delta}b_{02}.
\end{equation*}
\begin{proposition}\label{tuilun1}
	Let $f,g$ be defined in \eqref{300}, \eqref{301}. It holds
	\begin{align*}
	&\nabla \cdot f=0, \quad \quad \nabla \times f=\sqrt{-\Delta} f,
	\\
	& \nabla \cdot g=0, \quad \quad \nabla \times g=\sqrt{-\Delta} g,
	\\
	&|\nabla^k f| \leqslant \frac{|\alpha_1| M_2}{1+|x|}e^{-\frac{\nu t}{4}}, \quad |\nabla^k g| \leqslant \frac{|\alpha_2| M_2}{1+|x|}e^{-\frac{\mu t}{4}}, \quad 0\leqslant |k| \leqslant 5.
	\end{align*}
\end{proposition}
\begin{proof}
By
\begin{align*}
& \nabla \cdot v_0=0, \quad \nabla \times v_0=\sqrt{-\Delta}v_0,
\\
& f=e^{\nu t\Delta}u_{02}, \quad g=e^{\mu t\Delta}b_{02},
\end{align*}
we can deduce that
\begin{align*}
&\nabla \cdot f=0, \quad \quad \nabla \times f=\sqrt{-\Delta} f,
	\\
	& \nabla \cdot g=0, \quad \quad \nabla \times g=\sqrt{-\Delta} g.
\end{align*}
We choose a $C^\infty(\mathbb{R}^3)$ cut-off function $\alpha(\xi)$ such that $\alpha\equiv1$ on the support of $v_0$, and $\alpha(\xi)\equiv0$ if $|\xi|\geqslant1+2\delta$ or $|\xi|\leqslant1-2\delta$. Then we have
$$
f(t,x)=\alpha_1e^{-\frac{\nu t}{2}} {\mathcal{F}}^{-1}\left(e^{-\nu(|\xi|^2-\frac{1}{2})t}\alpha(\xi)\right) \ast v_0,
$$
$$
g(t,x)=\alpha_2e^{-\frac{\mu t}{2}} {\mathcal{F}}^{-1}\left(e^{-\mu(|\xi|^2-\frac{1}{2})t}\alpha(\xi)\right) \ast v_0.
$$
In a result, we get
$$|\nabla^k f| \leqslant \frac{|\alpha_1| M_2}{1+|x|}e^{-\frac{\nu t}{4}}, \quad |\nabla^k g| \leqslant \frac{|\alpha_2| M_2}{1+|x|}e^{-\frac{\mu t}{4}}, \quad 0\leqslant |k| \leqslant 5.$$
\end{proof}
\begin{proposition}\label{tuilun2}
	Set $\tilde{f}:=\chi_{M_0}f, \tilde{g}:=\chi_{M_0}g$. Let $f,g, \chi_{M_0}$ be defined by (2.2), (2.3) and (1.3)--(1.4) respectively. Then we have
	\begin{align*}
	& ||\tilde{f}||_{W^{5,\infty}}+||\tilde{g}||_{W^{5,\infty}} \leqslant \alpha_1M_1 e^{-\frac{\nu t}{4}}+\alpha_2M_1 e^{-\frac{\mu t}{4}},
	\\
	& ||\tilde{f} \times \left( \nabla \times \tilde{f}\right)||_{H^3}
	+||\tilde{g} \times \left( \nabla \times \tilde{g}\right)||_{H^3} \leqslant
	\left( \alpha_1^2 e^{-\frac{\nu t}{2}}+\alpha_2^2 e^{-\frac{\mu t}{2}}\right)\left( \delta M_0^{\frac{3}{2}}M_1^2+M_0^{-1}M_2^2\right)
	,
	\\
	&\int^\infty_0 || \tilde{f} \times \tilde{g}||_{H^3}(t)dt \leqslant CM_0^{\frac{3}{2}}M_1^2 \delta.
	\end{align*}
\end{proposition}
\begin{proof}
Firstly, we have $|\nabla^k \chi_{M_0}| \lesssim \frac{1}{M^k_0}, k\leqslant 5$. Then
\begin{align}\label{200}
||\tilde{f}||_{W^{5,\infty}}&=||\chi_{M_0}f||_{W^{5,\infty}} \leqslant ||\chi_{M_0}||_{W^{5,\infty}}||f||_{W^{5,\infty}}
 \lesssim ||f||_{W^{5,\infty}}.
\end{align}
Using $\hat{f}=e^{-\nu t |\xi|^2} \hat{u}_{02}$ and $ \text{supp}  \ \hat{u}_{02} \subseteq \left\{ \xi| 1-\delta \leqslant |\xi| \leqslant 1+\delta \right\}$, $0 < \delta \leqslant \frac{1}{2}$, we get
\begin{align*}
||f||_{W^{5,\infty}} \lesssim \left| \left| (1+|\xi|)^5 \hat {f}   \right| \right|_{L^1_\xi} \lesssim ||  \hat {f}  ||_{L^1_\xi} \lesssim \left| \left|  e^{-\nu t |\xi|^2} \hat{u}_{02}  \right| \right|_{L^1_\xi}  \lesssim \alpha_1 M_1 e^{-\frac{\nu t} {4}}.
\end{align*}
Similarly, we have
\begin{align}\label{201}
||g||_{W^{5,\infty}} \lesssim  \alpha_2 M_1 e^{-\frac{\mu t} {4}}.
\end{align}
Adding \eqref{201} to \eqref{200}, we obtain
\begin{align}\label{202}
||\tilde{f}||_{W^{5,\infty}}+||\tilde{g}||_{W^{5,\infty}} \leqslant \alpha_1M_1 e^{-\frac{\nu t}{4}}+\alpha_2M_1 e^{-\frac{\mu t}{4}}.
\end{align}
Secondly, we notice the fact
$$\nabla \times (\chi_{M_0}f)=\nabla \chi_{M_0} \times f +\chi_{M_0} \nabla \times f,$$
$$\nabla \times (\chi_{M_0}g)=\nabla \chi_{M_0} \times g +\chi_{M_0} \nabla \times g.$$
Thus, we get
\begin{align*}
&|| \tilde{f} \times (\nabla \times \tilde{f})||_{H^3} + || \tilde{g} \times (\nabla \times \tilde{g})||_{H^3}
\\
& \quad=|| \chi_{M_0}f \times \left(\nabla \times (\chi_{M_0}f)\right)||_{H^3}+|| \chi_{M_0}g \times \left(\nabla \times (\chi_{M_0}g)\right)||_{H^3}
\\
& \quad \lesssim || \chi^2_{M_0}||_{H^3} \left(  ||f \times (\nabla \times f)||_{W^{3,\infty}} + ||g \times (\nabla \times g)||_{W^{3,\infty}}\right)
\\
& \quad \quad + ||\nabla(\chi^2_{M_0})||_{W^{3,\infty}} \left(  || |f|^2 ||_{H^3} + || |g|^2||_{H^3}\right).
\end{align*}
We calculate that
\begin{equation}\label{204}
|| \chi^2_{M_0}||_{H^3} \lesssim \sum_{i=0}^{3}M_0^{-i}M_0^{\frac{3}{2}} \lesssim M_0^{\frac{3}{2}}, \quad ||\nabla(\chi^2_{M_0})||_{W^{3,\infty}} \lesssim \sum_{i=0}^{3}M_0^{-i-1}\lesssim M_0^{-1}.
\end{equation}
For $f\times f=0, \ g \times g=0$, then we have
\begin{align}\label{205}
&||f \times (\nabla \times f)||_{W^{3,\infty}} + ||g \times (\nabla \times g)||_{W^{3,\infty}}
\nonumber
\\
&\quad =||f \times (\nabla \times f-f)||_{W^{3,\infty}} + ||g \times (\nabla \times g-g)||_{W^{3,\infty}}
\nonumber
\\
& \quad \lesssim ||f||_{W^{3,\infty}} ||\nabla \times f-f||_{W^{3,\infty}}+||g||_{W^{3,\infty}}||\nabla \times g-g||_{W^{3,\infty}}
\nonumber
\\
& \quad \lesssim ||(1+|\xi|)^3 \hat{f} ||_{L^1_\xi} ||(1+|\xi|)^3 (|\xi|-1)\hat{f} ||_{L^1_\xi} + ||(1+|\xi|)^3 \hat{g} ||_{L^1_\xi} ||(1+|\xi|)^3 (|\xi|-1)\hat{g} ||_{L^1_\xi}
\nonumber
\\
& \quad \lesssim \delta \left( ||\hat{f} ||_{L^1_\xi}^2+||\hat{g} ||^2_{L^1_\xi} \right)
\nonumber
\\
& \quad \lesssim \delta \left( \alpha_1^2 e^{-\frac{\nu t}{2}}+\alpha_2^2 e^{-\frac{\mu t}{2}}\right) M_1^2.
\end{align}
Noticing that $\text{supp} \ \widehat{|f|^2}, \text{supp}\ \widehat{|g|^2} \subseteq \left\{ \xi| |\xi| \leqslant 2+2\delta  \right\}$, $0 < \delta \leqslant \frac{1}{2}$, we derive that
\begin{align}\label{206}
|| |f|^2 ||_{H^3} + || |g|^2 ||_{H^3}& \lesssim || |f|^2 ||_{L^2} + || |g|^2 ||_{L^2}
\nonumber
\\
& \lesssim ||f||^2_{L^4}+||g||^2_{L^4}
\nonumber
\\
& \lesssim  \alpha_1^2 e^{-\frac{\nu t}{2}} M_2^2 + \alpha_2^2 e^{-\frac{\mu t}{2}} M_2^2.
\end{align}
Combining \eqref{204}, \eqref{205} and \eqref{206}, we get
$$ ||\tilde{f} \times \left( \nabla \times \tilde{f}\right)||_{H^3}
	+||\tilde{g} \times \left( \nabla \times \tilde{g}\right)||_{H^3} \leqslant
	\left( \alpha_1^2 e^{-\frac{\nu t}{2}}+\alpha_2^2 e^{-\frac{\mu t}{2}}\right)\left( \delta M_0^{\frac{3}{2}}M_1^2+M_0^{-1}M_2^2\right) .$$
In what follows, we will estimate $\int^t_0 || \tilde{f} \times \tilde{g}||_{H^3}(\tau)d\tau.$
On one hand,
\begin{equation}\label{207}
|| \tilde{f} \times \tilde{g}||_{H^3}= || \chi_{M_0}f \times (\chi_{M_0}g)||_{H^3} \lesssim ||\chi_{M_0}^2||_{H^3} ||f \times g||_{W^{3, \infty}}.
\end{equation}
On the other hand, $\text{supp} \ \widehat{f \times g} \subseteq \left\{ \xi| |\xi| \leqslant 2+2\delta  \right\}$, $0 < \delta \leqslant \frac{1}{2}$. Then we have
\begin{equation}\label{208}
|| \tilde{f} \times \tilde{g}||_{H^3} \lesssim M_0^{\frac{3}{2}} ||\widehat{f \times g} ||_{L^1_\xi}.
\end{equation}
Calculate
\begin{align*}
\widehat{f \times g} & = \alpha_1 \alpha_2 \int_{\mathbb{R}^3} e^{-\nu|\xi-\eta|^2 t} \hat{v_0}(\xi-\eta) \times e^{-\mu|\eta|^2 t} \hat{v_0}(\eta)d\eta
\\
&= \frac{1}{2}\alpha_1 \alpha_2 \int_{\mathbb{R}^3} \left( e^{-(\nu|\xi-\eta|^2+\mu|\eta|^2) t}-e^{-(\mu|\xi-\eta|^2+\nu|\eta|^2) t} \right) \hat{v_0}(\xi-\eta) \times  \hat{v_0}(\eta)d\eta,
\end{align*}
and
\begin{equation}\label{209}
		\begin{split}
			& \left| e^{-(\nu|\xi-\eta|^2+\mu|\eta|^2) t}-e^{-(\mu|\xi-\eta|^2+\nu|\eta|^2) t} \right|\\
			& \quad =e^{-\mu(|\xi-\eta|^2+|\eta|^2)t}\left |e^{-(\nu-\mu)|\xi-\eta|^2t}-e^{-(\nu-\mu)|\eta|^2t} \right|\\
			 & \quad \leqslant C t e^{-\mu(|\xi-\eta|^2+|\eta|^2)t} \left| |\xi-\eta|^2-|\eta|^2\right|\\
			& \quad \leqslant C e^{-\frac{\mu}{2}(|\xi-\eta|^2+|\eta|^2)t}\frac{||\xi-\eta|^2-|\eta|^2|}{|\xi-\eta|^2+|\eta|^2}.
		\end{split}
\end{equation}
	\qquad In the support of  $\hat{v_0}(\xi-\eta) \times  \hat{v_0}(\eta)$, we have
	\begin{equation}\label{702}
	\frac{||\xi-\eta|^2-|\eta|^2|}{|\xi-\eta|^2+|\eta|^2} \leqslant 10 \delta.
	\end{equation}
Therefore, we conclude that
$$ \int^\infty_0 || \tilde{f} \times \tilde{g}||_{H^3}(t)dt \leqslant CM_0^{\frac{3}{2}}M_1^2 \delta .$$ Then we complete the proof of Proposition \ref{tuilun2}.
\end{proof}
\section{The proof of Theorem \ref{thm} }
In this section, we will prove Theorem \ref{thm} using a perturbation argument along with a standard cut-off technique.
\begin{proof}[\\Proof of Theorem \ref{thm}]
Let $\tilde{f}=\chi_{M_0}f, \tilde{g}=\chi_{M_0}g$, and $u=U+\tilde{f}, b=B+\tilde{g}$. Then $U,B$ satisfy
\begin{align}\label{U}
U_t-\nu \Delta U+\nabla\left(p+\frac{1}{2}|\tilde{f}|^2-\frac{1}{2}|\tilde{g}|^2 \right)&=-U \cdot \nabla U-\tilde{f} \cdot \nabla U - U \cdot \nabla \tilde{f}
\nonumber
\\ & \quad +B \cdot \nabla B+ \tilde{g} \cdot \nabla B+B \cdot \nabla \tilde{g}+F,
\end{align}
\begin{align}\label{B}
B_t-\mu \Delta B&=-U \cdot \nabla B-\tilde{f} \cdot \nabla B - U \cdot \nabla \tilde{g}+B \cdot \nabla U+ \tilde{g} \cdot \nabla U+ B \cdot \nabla \tilde{f}
\nonumber
\\ & \quad -\eta \nabla \times \left( \left(\nabla \times B \right) \times B\right)-\eta \nabla \times \left( \left(\nabla \times B \right) \times \tilde{g}\right)
\nonumber
\\
& \quad - \eta \nabla \times \left( \left(\nabla \times \tilde{g} \right) \times B\right)-\eta \nabla \times \left( (\nabla \times \tilde{g})\times \tilde{g}\right)+G,
\end{align}
where
$$F:= \tilde{f}\times (\nabla \times \tilde{f})-\tilde{g}\times \left(\nabla \times \tilde{g}\right)-\nu \Delta\chi_{M_0}f+2\nu \nabla \cdot
\left( \nabla\chi_{M_0} f\right),$$
$$G:= \nabla \times ( \tilde{f} \times \tilde{g})-\mu \Delta\chi_{M_0}g+2\mu \nabla \cdot
\left( \nabla\chi_{M_0} g\right)+\frac{1}{2}f \cdot \nabla {\chi_{M_0}}^2g-\frac{1}{2}g \cdot \nabla {\chi_{M_0}}^2f.$$
In what follows, we will derive some energy estimates of $U,B$.

\textbf{\\Step 1: Energy inequalities of $B$.\\}
Operating Equation \eqref{B} with $\partial^\alpha, |\alpha|\leqslant 3$, and taking $L^2$ inner product with $\partial^\alpha B$, we get
\begin{align}\label{307}
&\frac{1}{2}\frac{d}{dt}||\partial^\alpha B||^2_{L^2}+\mu ||\partial^\alpha \nabla B||^2_{L^2}
\nonumber
\\
&\quad= -\sum_{1 \leqslant |\beta|\leqslant |\alpha|}
{\alpha \choose \beta} \left( \int_{\mathbb{R}^3} \partial^\beta U \cdot \nabla\partial^{\alpha-\beta} B  \partial^\alpha B dx + \int_{\mathbb{R}^3}\partial^\beta \tilde{f} \cdot \nabla \partial^{\alpha-\beta}B \partial^\alpha B dx  \right)
\nonumber
\\
& \quad \quad - \int_{\mathbb{R}^3}\partial^\alpha \left(U \cdot \nabla \tilde{g} \right) \partial^\alpha Bdx+\sum_{1 \leqslant |\beta|\leqslant |\alpha|}{\alpha \choose \beta}\int_{\mathbb{R}^3}\partial^\beta B \cdot \nabla  \partial^{\alpha-\beta}U \partial^\alpha B dx
\nonumber
\\
& \quad \quad + \sum_{1 \leqslant |\beta|\leqslant |\alpha|}{\alpha \choose \beta} \int_{\mathbb{R}^3}\partial^\beta  \tilde{g} \cdot \nabla \partial^{\alpha-\beta} U \partial^\alpha B dx + \int_{\mathbb{R}^3}\partial^\alpha \left(B \cdot \nabla \tilde{f} \right) \partial^\alpha B dx
\nonumber
\\
& \quad \quad +\int_{\mathbb{R}^3}\partial^\alpha G \partial^\alpha Bdx+T_1+T_2+I,
\end{align}
where
\begin{align}\label{308}
T_1&=-\int_{\mathbb{R}^3}U\cdot \nabla \partial^\alpha B \partial^\alpha Bdx-\int_{\mathbb{R}^3}\tilde{f}\cdot \nabla \partial^\alpha B \partial^\alpha Bdx
\nonumber
\\
&=-\int_{\mathbb{R}^3}u \cdot \nabla \partial^\alpha B \partial^\alpha Bdx
\nonumber
\\
&=0,
\end{align}
\begin{align}\label{309}
T_2&=\int_{\mathbb{R}^3}B \cdot \nabla \partial^\alpha U \partial^\alpha Bdx+\int_{\mathbb{R}^3}\tilde{g}\cdot \nabla \partial^\alpha U \partial^\alpha Bdx
\nonumber
\\
&=\int_{\mathbb{R}^3}b \cdot \nabla \partial^\alpha U \partial^\alpha Bdx,
\end{align}
\begin{align}\label{000}
I&=-\eta\int_{\mathbb{R}^3} \partial^\alpha B \cdot \partial^\alpha \left( \nabla \times \left( \left(\nabla \times B \right) \times B\right) \right)dx
\nonumber
\\
& \quad -\eta\int_{\mathbb{R}^3} \partial^\alpha B \cdot \partial^\alpha \left( \nabla \times \left( \left(\nabla \times B \right) \times \tilde{g}\right) \right)dx
\nonumber
\\
& \quad -\eta\int_{\mathbb{R}^3} \partial^\alpha B \cdot \partial^\alpha \left( \nabla \times \left( \left(\nabla \times \tilde{g} \right) \times B\right) \right)dx
\nonumber
\\
& \quad -\eta\int_{\mathbb{R}^3} \partial^\alpha B \cdot \partial^\alpha \left( \nabla \times \left( \left(\nabla \times \tilde{g} \right) \times \tilde{g}\right) \right)dx
\nonumber
\\
&:=I_1+I_2+I_3+I_4.
\end{align}
Firstly, we have
\begin{align}\label{3007}
\left| \int_{\mathbb{R}^3}\partial^\alpha G \partial^\alpha B dx \right| &=\big| \int_{\mathbb{R}^3} \partial^\alpha (\nabla \times ( \tilde{f} \times \tilde{g})+2\nu\nabla \cdot \left( \nabla\chi_{M_0}g\right)-\nu\Delta \chi_{M_0}g
\\
\nonumber
& \quad \quad+\frac{1}{2}f \cdot \nabla \chi^2_{M_0}g-\frac{1}{2}g \cdot \nabla \chi^2_{M_0}f ) \cdot \partial^\alpha B dx\big|
\nonumber
\\
&  \lesssim ||\tilde{f}\times \tilde{g}||_{H^3}||\nabla B||_{H^3}+||\nabla\chi_{M_0}g||_{H^3}||\nabla B||_{H^3}
\nonumber
\\
& \quad +\left(||f \cdot \nabla \chi^2_{M_0}g||_{W^{3,\frac{6}{5}}}+||g \cdot \nabla \chi^2_{M_0}f||_{W^{3,\frac{6}{5}}}\right)||B||_{W^{3,6}}
\nonumber
\\
& \quad +||\Delta \chi_{M_0}g||_{W^{3,\frac{6}{5}}}||B||_{W^{3,6}}
\nonumber
\\
& \lesssim \left(||\tilde{f}\times \tilde{g}||_{H^3}+M_0^{-1}||g||_{H^3(|x|\leqslant 2M_0)}  \right)||\nabla B||_{H^{3}}
\nonumber
\\
& \quad + \left(M_0^{-2}||g||_{W^{3,\frac{6}{5}}(|x|\leqslant 2M_0)} +M^{-1}_0 ||f\otimes g||_{W^{3,\frac{6}{5}}(|x|\leqslant 2M_0)} \right)||\nabla B||_{H^{3}}
\nonumber
\\
& \lesssim \left( ||\tilde{f}\times \tilde{g}||_{H^3}+ M_0^{-\frac{1}{2}}M_2 |\alpha_2| e^{-\frac{\mu t}{4}} +|\alpha_1\alpha_2| M_0^{-\frac{1}{2}} M_2^2 e^{-\frac{(\mu
 +\nu)t}{4}}\right) ||\nabla B||_{H^{3}}
\nonumber
.
\end{align}

Next, we need to estimate $I_1, I_2, I_3$ and $I_4$. For
\begin{align*}
I_1&=-\eta\int_{\mathbb{R}^3} \partial^\alpha B \cdot \partial^\alpha \left( \nabla \times \left( \left(\nabla \times B \right) \times B\right) \right)dx
\\
&= \eta\int_{\mathbb{R}^3} \partial^\alpha (\nabla \times B) \cdot \partial^\alpha  \left( \left(\nabla \times B \right) \times B\right) dx
\\
&= \eta\int_{\mathbb{R}^3} \left\{  \partial^\alpha  \left( \left(\nabla \times B \right) \times B\right)  -  \partial^\alpha (\nabla \times B ) \times B  \right\}\cdot \partial^\alpha (\nabla \times B) dx,
\end{align*}
using Lemma \ref{lemc}, we deduce that
\begin{align}\label{310}
|I_1| &\leqslant \eta\sum_{|\alpha|\leq 3} ||\partial^\alpha (\nabla \times B)||_{L^2} ||\partial^\alpha  \left( \left(\nabla \times B \right) \times B\right)  -  \partial^\alpha (\nabla \times B ) \times B   ||_{L^2}
\nonumber
\\
&\lesssim \eta ||\nabla B||_{H^3}\left( ||\nabla \times B||_{H^2}||\nabla B||_{L^\infty}+||\nabla \times B||_{L^\infty}||B||_{H^3}\right)
\nonumber
\\
&\lesssim \eta||\nabla B||^2_{H^3}|| B||_{H^3}.
\end{align}
We calculate
\begin{align*}
I_2&=-\eta\int_{\mathbb{R}^3} \partial^\alpha B \cdot \partial^\alpha \left( \nabla \times \left( \left(\nabla \times B \right) \times \tilde{g}\right) \right)dx
\\
&=\eta\int_{\mathbb{R}^3} \partial^\alpha (\nabla \times B) \cdot \partial^\alpha  \left( \left(\nabla \times B \right) \times \tilde{g}\right)dx
\\
&=\eta\sum_{1 \leqslant |\beta|\leqslant |\alpha|}{\alpha \choose \beta}\int_{\mathbb{R}^3} \partial^\alpha (\nabla \times B) \cdot \left( \partial^{\alpha-\beta}(\nabla \times B) \times (\partial^\beta \tilde{g})\right)dx
\\
& \quad+\eta\int_{\mathbb{R}^3} \partial^\alpha (\nabla \times B) \cdot \left( \partial^{\alpha}(\nabla \times B) \times  \tilde{g}\right)dx
\\
&=\eta\sum_{1 \leqslant |\beta| \leqslant |\alpha|}{\alpha \choose \beta}\int_{\mathbb{R}^3} \partial^\alpha (\nabla \times B) \cdot \left( \partial^{\alpha-\beta}(\nabla \times B) \times (\partial^\beta \tilde{g})\right)dx,
\end{align*}
and $$\int_{\mathbb{R}^3} \partial^\alpha (\nabla \times B) \cdot \left( \partial^{\alpha}(\nabla \times B) \times  \tilde{g}\right)dx=0.$$
Therefore, we have
\begin{equation}\label{311}
|I_2| \lesssim \eta||\nabla B||_{H^3}||\nabla B||_{H^2}||\tilde{g}||_{W^{3,\infty}} \lesssim \eta||\nabla B||_{H^3} || B||_{H^3} ||\tilde{g}||_{W^{3,\infty}}.
\end{equation}
To estimate $I_3$, we divide it into two parts
\begin{align*}
I_3&=-\eta\int_{\mathbb{R}^3} \partial^\alpha B \cdot \partial^\alpha \left( \nabla \times \left( \left(\nabla \times \tilde{g} \right) \times B\right) \right)dx
\\
&=-\eta\sum_{|\beta| \leqslant |\alpha|-1}{\alpha \choose \beta}\int_{\mathbb{R}^3} \partial^\alpha B \cdot \nabla \times \left( \partial^{\alpha-\beta} (\nabla \times \tilde{g}) \times \partial^\beta B\right)dx
\\
& \quad \quad-\eta\int_{\mathbb{R}^3} \partial^\alpha B \cdot \nabla \times \left(  (\nabla \times \tilde{g}) \times \partial^\alpha B\right)dx
:=I_{31}+I_{32}.
\end{align*}
Applying H\"older inequality, we easily infer
\begin{equation}\label{312}
|I_{31}| \lesssim \eta ||B||^2_{H^3} || \tilde{g}||_{W^{5,\infty}}.
\end{equation}
Moreover,
\begin{align*}
I_{32}&=-\eta \int_{\mathbb{R}^3} \partial^\alpha B \cdot \nabla \times \left(  (\nabla \times \tilde{g}) \times \partial^\alpha B\right)dx
\\
&=\eta \int_{\mathbb{R}^3} \partial^\alpha B \cdot \left( (\nabla \times \tilde{g}) \cdot \nabla \right) \partial^\alpha Bdx-\eta\int_{\mathbb{R}^3} \partial^\alpha B \cdot \left( \partial^\alpha B  \cdot \nabla\right)   (\nabla \times \tilde{g}) dx
\\
&= \frac{\eta}{2}\int_{\mathbb{R}^3} \nabla \cdot  \left( (\nabla \times \tilde{g}) |\partial^\alpha B|^2 \right) dx-\frac{\eta}{2}\int_{\mathbb{R}^3} |\partial^\alpha B|^2 \nabla \cdot \left( \nabla \times \tilde{g} \right) dx
\\
& \quad -\eta \int_{\mathbb{R}^3} \partial^\alpha B \cdot \left( \partial^\alpha B  \cdot \nabla\right)   (\nabla \times \tilde{g}) dx
\\
&=-\frac{1}{2}\eta\int_{\mathbb{R}^3} |\partial^\alpha B|^2 \nabla \cdot \left( \nabla \times \tilde{g} \right) dx
 -\eta \int_{\mathbb{R}^3} \partial^\alpha B \cdot \left( \partial^\alpha B  \cdot \nabla\right)   (\nabla \times \tilde{g}) .
\end{align*}
thus,
\begin{equation}\label{3121}
|I_{32}| \lesssim \eta ||\tilde{g}||_{W^{2,\infty}} ||B||^2_{H^3}.
\end{equation}
Combining \eqref{312} and \eqref{3121}, we get
\begin{equation}\label{3122}
	|I_{3}| \lesssim \eta ||\tilde{g}||_{W^{5,\infty}} ||B||^2_{H^3}.
\end{equation}
As for $I_4$, we can write
\begin{align*}
I_4&=-\eta \int_{\mathbb{R}^3} \partial^\alpha B \cdot \partial^\alpha \left( \nabla \times \left( \left(\nabla \times \tilde{g} \right) \times \tilde{g}\right) \right)dx
\\
&=\eta \int_{\mathbb{R}^3} \partial^\alpha (\nabla \times B) \cdot \partial^\alpha \left( \left(\nabla \times \tilde{g} \right) \times \tilde{g} \right)dx.
\end{align*}
We thus have
\begin{equation}\label{313}
|I_4| \lesssim \eta ||\nabla B||_{H^3}  ||\left(\nabla \times \tilde{g} \right) \times \tilde{g}||_{H^3}.
\end{equation}
\textbf{\\Step 2: Energy inequalities of $U$.\\}
Operating Equation \eqref{U} with $\partial^\alpha, |\alpha|\leqslant 3$, and taking $L^2$ on Equation \eqref{U} yields
\begin{align}\label{314}
&\frac{1}{2}\frac{d}{dt}||\partial^\alpha U||^2_{L^2}+\mu ||\partial^\alpha \nabla U||^2_{L^2}
\nonumber
\\
&\quad= -\sum_{1 \leqslant |\beta| \leqslant |\alpha|}{\alpha \choose \beta} \left(\int_{\mathbb{R}^3}\partial^\beta U \cdot \nabla \partial^{\alpha-\beta}U \partial^\alpha U dx + \int_{\mathbb{R}^3}\partial^\beta \tilde{f} \cdot \nabla \partial^{\alpha-\beta} U \partial^\alpha U dx  \right)
\nonumber
\\
& \quad \quad - \int_{\mathbb{R}^3}\partial^\alpha \left(U \cdot \nabla \tilde{f} \right) \partial^\alpha U dx+\sum_{1 \leqslant |\beta|\leqslant |\alpha|}{\alpha \choose \beta}\int_{\mathbb{R}^3}\partial^\beta B \cdot \nabla \partial^{\alpha-\beta} B  \partial^\alpha U dx
\nonumber
\\
& \quad \quad + \sum_{1 \leqslant |\beta|\leqslant |\alpha|}{\alpha \choose \beta}\int_{\mathbb{R}^3}\partial^\beta  \tilde{g} \cdot \nabla \partial^{\alpha-\beta} B \partial^\alpha U dx + \int_{\mathbb{R}^3}\partial^\alpha \left(B \cdot \nabla \tilde{g} \right) \partial^\alpha U dx
\nonumber
\\
& \quad \quad +\int_{\mathbb{R}^3}\partial^\alpha F \partial^\alpha Udx-\int_{\mathbb{R}^3}\partial^\alpha \nabla \left(p+\frac{1}{2}|\tilde{f}|^2-\frac{1}{2}|\tilde{g}|^2 \right)\partial^\alpha Udx+T_3+T_4,
\end{align}
where
\begin{align}\label{001}
T_3&:=-\int_{\mathbb{R}^3}U \cdot \nabla \partial^\alpha U \partial^\alpha Udx-\int_{\mathbb{R}^3}\tilde{f}\cdot \nabla \partial^\alpha U \partial^\alpha Udx
\nonumber
\\
&=-\int_{\mathbb{R}^3}u \cdot \nabla \partial^\alpha U \partial^\alpha Udx
\nonumber
\\
&=0,
\end{align}
and
\begin{align*}
T_4&:=\int_{\mathbb{R}^3}B \cdot \nabla \partial^\alpha B \partial^\alpha Udx+\int_{\mathbb{R}^3}\tilde{g}\cdot \nabla \partial^\alpha B \partial^\alpha Udx
\\
&=\int_{\mathbb{R}^3}b \cdot \nabla \partial^\alpha B \partial^\alpha Udx.
\end{align*}
Combining $T_2$ and $T_2$,  we have
\begin{equation}\label{3051}
T_2+T_4=0.
\end{equation}
Firstly, we have
\begin{align}\label{30000}
\left| \int_{\mathbb{R}^3}\partial^\alpha F \partial^\alpha U dx \right| &=\big| \int_{\mathbb{R}^3} \partial^\alpha (\tilde{f}\times (\nabla \times \tilde{f})-\tilde{g}\times \left(\nabla \times \tilde{g}\right)-\nu \Delta\chi_{M_0}f
\\
\nonumber
& \quad \quad+2\nu \nabla \cdot
\left( \nabla\chi_{M_0} f\right) ) \cdot \partial^\alpha U dx\big|
\nonumber
\\
&  \lesssim \left( ||\tilde{f}\times (\nabla \times \tilde{f})||_{H^3}+||\tilde{g}\times (\nabla \times \tilde{g})||_{H^3} \right)||\nabla U||_{H^3}
\nonumber
\\
& \quad +|| \nabla\chi_{M_0}f ||_{H^{3}} ||\nabla U||_{H^{3}}
+||\Delta \chi_{M_0}f||_{W^{3,\frac{6}{5}}}||U||_{W^{3,6}}
\nonumber
\\
& \lesssim \left( ||\tilde{f}\times (\nabla \times \tilde{f})||_{H^3}+||\tilde{g}\times (\nabla \times \tilde{g})||_{H^3} \right)||\nabla U||_{H^3}
\nonumber
\\
&\quad \quad +M_0^{-\frac{1}{2}}M_2 |\alpha_1| e^{-\frac{\nu t}{4}}  ||\nabla U||_{H^{3}}
\nonumber
.
\end{align}
Since
\begin{align*}
p&=(-\Delta)^{-1} \text{div} \left( u \cdot \nabla u-b \cdot \nabla b \right)
\\
&=\sum_{i,j}(-\Delta)^{-1} \partial_i \partial_j \left(u_iU_j-b_iB_j \right)+(-\Delta)^{-1} \nabla \cdot \left(U \cdot \nabla \tilde{f}-B\cdot \nabla \tilde{g} \right)
\\
& \quad + (-\Delta)^{-1} \nabla \cdot \left( \tilde{f}\times (\nabla \times \tilde{f})-\tilde{g}\times \left(\nabla \times \tilde{g}\right) \right)-\frac{1}{2}|\tilde{f}|^2+\frac{1}{2}|\tilde{g}|^2,
\end{align*}
then we have
\begin{align*}
\Pi&:=\left|-\sum_{|\alpha|\leqslant 3}\int_{\mathbb{R}^3}\partial^\alpha \nabla \left(p+\frac{1}{2}|\tilde{f}|^2-\frac{1}{2}|\tilde{g}|^2 \right)\partial^\alpha Udx \right|
\nonumber
\\
&  \leq \left|\sum_{|\alpha|\leqslant 3} \int_{\mathbb{R}^3} \partial^\alpha \sum_{i,j}(-\Delta)^{-1} \partial_i \partial_j(u_iU_j-b_iB_j)\partial^\alpha \nabla \cdot Udx \right|
\nonumber
\\
&  \quad +\left|\sum_{|\alpha|\leqslant 3} \int_{\mathbb{R}^3} \partial^\alpha \nabla (-\Delta)^{-1} \nabla \partial \cdot \left( U\cdot \nabla \tilde{f}-B\cdot \nabla \tilde{g}\right)\partial^\alpha Udx \right|
\nonumber
\\
&  \quad +\left|\sum_{|\alpha|\leqslant 3} \int_{\mathbb{R}^3} \partial^\alpha \nabla (-\Delta)^{-1} \nabla  \cdot \left( \tilde{f}\times (\nabla \times \tilde{f})-\tilde{g}\times \left(\nabla \times \tilde{g}\right) \right) \partial^\alpha Udx \right|.
\end{align*}
Using H\"older inequality and Calderon-Zygmund estimate \cite{S}, then we get
\begin{align*}
\Pi &  \leqslant \left( || u \otimes U||_{W^{3,\frac{3}{2}}} +|| h \otimes B||_{W^{3,\frac{3}{2}}}\right) ||f \cdot \nabla\chi_{M_0}||_{W^{3,3}}
\nonumber
\\
&  \quad
+ ||U\cdot \nabla \tilde{f}-B\cdot \nabla \tilde{g}||_{H^3}||U||_{H^3}
\nonumber
\\
&  \quad
+  ||\tilde{f}\times (\nabla \times \tilde{f})-\tilde{g}\times \left(\nabla \times \tilde{g}\right)||_{H^3}||U||_{H^3}
\nonumber
\\
& \leqslant \left( ||U||^2_{H^3}+||\tilde{f}||_{W^{3,6}}||U||_{H^3}+||H||^2_{H^3}\right)||\nabla \chi_{M_0}||_{W^{3,\infty}} ||f||_{W^{3,3}(M_0 \leqslant |x| \leqslant 2M_0)}
\nonumber
\\
&   \quad +||\tilde{g}||_{W^{3,2}} ||H||_{W^{3,6}}||\nabla \chi_{M_0}||_{W^{3,\infty}} ||f||_{W^{3,3}(M_0 \leqslant |x|\leqslant 2M_0)}
\nonumber
\\
&   \quad + \left( ||\nabla \tilde{f}||_{W^{3,\infty}}||U||_{H^3}+||\nabla \tilde{g}||_{W^{3,\infty}}||B||_{H^3}\right)
\nonumber
\\
&  \quad +\left( ||\tilde{f}\times (\nabla \times \tilde{f})||_{H^3}+||\tilde{g}\times \left(\nabla \times \tilde{g}\right)||_{H^3} \right)||U||_{H^3}
\nonumber
\\
& \quad +||\nabla B||_{H^3}  ||\left(\nabla \times \tilde{g} \right) \times \tilde{g}||_{H^3}+||B||^2_{H^3}||\tilde{g}||_{W^{5,\infty}},
\nonumber
\\
&  \leqslant \left( ||U||^2_{H^3} +\alpha_1 M_2 e^{-\frac{\mu t}{4}}||U||_{H^3} + ||B||^2_{H^3} + \alpha_2 M_2 M_0^{\frac{1}{2}} e^{ \frac{\nu t}{4}} ||\nabla B||_{H^3}\right)\alpha_1 M_2 M_0^{-1}e^{-\frac{\mu t}{4}}
\nonumber
\\
&   \quad + \left( \alpha_1M_1 e^{-\frac{\mu t}{4}}||U||_{H^3}+\alpha_2 M_1 e^{-\frac{\nu t}{4}}||B||_{H^3}\right)||U||_{H^3}
\nonumber
\\
&   \quad +\left( \alpha_1^2 e^{-\frac{\nu t}{2}}+\alpha_2^2 e^{-\frac{\mu t}{2}}\right)\left( \delta M_0^{\frac{3}{2}}M_1^2+M_0^{-1}M_2^2 \right) ||U||_{H^3}.
\end{align*}
Therefore, we derive that
\begin{align}\label{315}
 \Pi &\lesssim \alpha_1(M_1+M_2)e^{-\frac{\mu t}{4}} ||U||^2_{H^3}+\alpha_1 M_2 e^{-\frac{\mu t}{4}} ||B||^2_{H^3}
\nonumber
\\
& \quad +
\alpha_2M_1 e^{-\frac{\nu t}{4}}||U||_{H^3}||B||_{H^3}
\nonumber
\\
& \quad +\left( \alpha_1^2 e^{-\frac{\nu t}{2}}+\alpha_2^2 e^{-\frac{\mu t}{2}}\right)\left( \delta M_0^{\frac{3}{2}}M_1^2+M_0^{-1}M_2^2 \right) ||U||_{H^3}
\nonumber
\\
& \quad + \alpha_1\alpha_2 M_0^{-\frac{1}{2}}M_2^2 e^{-\frac{\nu t+\mu t}{4}} ||\nabla B||_{H^3}.
\end{align}
\textbf{\\Step 3: Energy estimates of $U,B$.\\}
Gathering above estimates in Step 1 and Step 2, we obtain
\begin{align*}
	&\frac{1}{2}\frac{d}{dt}\left(||U||^2_{H^3}+||B||^2_{H^3}\right) + \frac{\nu}{2}||\nabla U||^2_{H^3}+ \frac{\mu}{2}||\nabla B||^2_{H^3}
\\
& \quad \lesssim \left(||U||_{H^3}+(1+\eta)||B||_{H^3}\right) \left( ||\nabla U||^2_{H^3}+ ||\nabla B||^2_{H^3}\right)
\\
& \quad \quad + \left(||\tilde{f}||_{W{4,\infty}}+(1+\eta)||\tilde{g}||_{W{4,\infty}}\right)
\left(||U||^2_{H^3}+||B||^2_{H^3}\right)
\\
& \quad \quad +\left( ||\tilde{f}\times \tilde{g}||_{H^3}+ M_0^{-\frac{1}{2}}M_2 |\alpha_2| e^{-\frac{\mu t}{4}} +|\alpha_1 \alpha_2| M_0^{-\frac{1}{2}} M_2^2 e^{-\frac{(\mu
 +\nu)t}{4}} \right) ||\nabla B||_{H^{3}}
\\
&\quad \quad+\eta || (\nabla \times \tilde{g}) \times \tilde{g} ||_{H^3} ||\nabla B||_{H^{3}}+|\alpha_1|(M_1+M_2)e^{-\frac{\mu t}{4}} ||U||^2_{H^3}
\\
& \quad \quad +|\alpha_1| M_2 e^{-\frac{\mu t}{4}} ||B||^2_{H^3}+
|\alpha_2|M_1 e^{-\frac{\nu t}{4}}||U||_{H^3}||B||_{H^3}
\nonumber
\\
& \quad \quad+\left( \alpha_1^2 e^{-\frac{\nu t}{2}}+\alpha_2^2 e^{-\frac{\mu t}{2}}\right)\left( \delta M_0^{\frac{3}{2}}M_1^2+M_0^{-1}M_2^2 \right) ||U||_{H^3}
\nonumber
\\
& \quad \quad+ |\alpha_1\alpha_2| M_0^{-\frac{1}{2}}M_2^2 e^{-\frac{\nu t+\mu t}{4}} ||\nabla B||_{H^3}+ \eta||\tilde{g}||_{W^{5,\infty}}||B||^2_{H^3}
\\
& \quad \quad+\left( ||\tilde{f}\times (\nabla \times \tilde{f})||_{H^3}+||\tilde{g}\times \left(\nabla \times \tilde{g}\right)||_{H^3} \right)||U||_{H^3}
\nonumber
\\
& \quad \quad +\eta ||\nabla B||_{H^3} ||B||_{H^3} ||\tilde{g}||_{W^{3,\infty}}++M_0^{-\frac{1}{2}}M_2 |\alpha_1| e^{-\frac{\nu t}{4}}  ||\nabla U||_{H^{3}}.
\end{align*}
Using Proposition \ref{tuilun2}, Young's inequality and $\delta^{-\frac{1}{2}} \geqslant M_0 \gg 1$, we derive that
\begin{align}\label{327}
&\frac{d}{dt}\left(||U||^2_{H^3}+||B||^2_{H^3}\right) + \left(\frac{\nu}{2}-C||U||_{H^3}-C||B||_{H^3} \right)||\nabla U||^2_{H^3}
\nonumber
\\
& \quad +\left(\frac{\mu}{2}-C||U||_{H^3}-C||B||_{H^3} \right)||\nabla B||^2_{H^3}
\nonumber
\\
& \quad \leqslant C\left( e^{-\frac{\nu t}{4}} +e^{-\frac{\mu t}{4}} \right)\left(||U||^2_{H^3}+||B||^2_{H^3}\right)+C\left( M_0^{-1}+\delta^2 M_0^3\right) \left( e^{-\frac{\nu t}{4}} +e^{-\frac{\mu t}{4}} \right)
\end{align}
for some constant $C$ depending on $M_1,M_2,\mu,\nu,\eta, \alpha_1,\alpha_2.$

For $t \in [0,\infty)$, we assume that
\begin{equation*}
||U(t)||^2_{H^3}+||B(t)||^2_{H^3} \leqslant \frac{\min\{ \mu, \nu\}}{4C}.
\end{equation*}
In case $t=0$, the above estimate holds.
Applying \eqref{327}, Gronwall's inequality and $\delta \leqslant M_0^{-2}$, we have
\begin{equation}\label{328}
||U(t)||_{H^3}+||B(t)||_{H^3} \lesssim M_0^{-\frac{1}{2}}.
\end{equation}
In a result,
\begin{equation}\label{329}
 ||U(t)||_{H^3}+||B(t)||_{H^3} \lesssim M_0^{-\frac{1}{2}}
\end{equation}
for all $t \in [0,\infty)$. Therefore, we complete the proof of Theorem \ref{thm}.
\end{proof}
\section*{Acknowledgement}

The author is supported by Changsha University of Science and Technology, P. R. China, the Education Department of Hunan Province, general Program(grant No. 17C0039).


\begin{thebibliography}{4}
	\addcontentsline{toc}{section}{References}
	
\bibitem{AD}
M. Acheritogaray, P. Degond, A. Frouvelle and J.-G. Liu, \textit{Kinetic formulation and global existence for the Hall-magneto-hydrodynamics system}. Kinet. Relat. Models 4 (2011), 901-918.
\bibitem{BT}
S. A. Balbus and C. Terquem, \textit{Linear analysis of the Hall effect in protostellar disks}. The Astrophysical Journal 552 (2001), 235-247.	
\bibitem{CDL}
D. Chae, P. Degond and J.G. Liu, \textit{Well-posedness for Hallmagnetohydrodynamics}. Ann. I. H. Poincar\'{e} 31 (2014), 555-565.
\bibitem{CL}
D. Chae and J. Lee, \textit{On the blow-up criterion and small data global existence for the Hall- magneto-hydrodynamics}. J. Differential Equations 256 (2014), 3835-3858.

\bibitem{CS}
D. Chae and M. Schonbek, \textit{On the temporal decay for the Hall-magnetohydrodynamic equations}. J. Differential Equations, 255 (2013), 3971-3982.
\bibitem{CWW}
D. Chae, R. Wan and J. Wu, \textit{Local well-posedness for the Hall-MHD equations with fractional magnetic diffusion}. J. Math. Fluid Mech. 17 (2015), 627–638.
\bibitem{CW}
D. Chae and S. Weng, \textit{Singularity formation for the incompressible Hall-MHD equations without resistivity}. Ann. Inst. Henri Poincar\'{e} Anal. Non Lin\'{e}aire, 33 (2016), 1009–1022.
\bibitem{CW2}
D. Chae and J. Wolf, \textit{On partial regularity for the 3D non-stationary Hall magnetohydrodynamics equations on the plane}. SIAM J. Math. Anal., 48 (2016), 443–469.
\bibitem{CG}
J.Y. Chemin and I. Gallagher, \textit{Well-posedness and stability results for the Navier-Stokes equa
tions in $R^3$}. Ann. Inst. H. H. Poincar\'{e} Anal. Non Lineaire 26 (2009), 599-624.
\bibitem{CM}
P. Constantin and A. Majda, \textit{The Beltrami spectrum for incompressible fluid flows}. Commun. Math. Phys. 115 (1988), 435-456.
\bibitem{D}
M.M. Dai, \textit{Local well-posedness for the Hall-MHD system in optimal Sobolev spaces}. arXiv:1803.08117.

\bibitem{F}
T. G. Forbes, \textit{Magnetic reconnection in solar flares}. Geophysical and astrophysical fluid dynamics 62 (1991), 15-36.

\bibitem{HG}
H. Homann and R. Grauer, \textit{Bifurcation analysis of magnetic reconnection in Hall-MHD systems}. Physica D 208 (2005), 59-72.


\bibitem{L}
M.J. Lighthill, \textit{Studies on magnetohydrodynamic waves and other anisotropic wave motions.
Philos}. Trans. R. Soc. Lond., Ser. A (1960), 397-430.
\bibitem{LLZ}
Z. Lei, F.H. Lin and Y. Zhou, \textit{Structure of helicity and global solutions of incompressible Navier-Stokes equation}. Arch. Ration. Mech. Anal. 218 (2015), 1417-1430..
\bibitem{Lin}
F.H. Lin, and P. Zhang, \textit{Global small solutions to an MHD-type system: the three-dimensional case}. Comm. Pure Appl. Math. 67 (2014), 531-580.

\bibitem{LZZ}
Y.R. Lin, H.L. Zhang and Y. Zhou, \textit{Global smooth solutions of MHD equations with large data}. J. Differential Equations 261 (2016), 102-112.

\bibitem{M}
P. A. Davidson, \textit{An Introduction to Magnetohydrodynamics}. Cambridge Texts in Applied Mathematics, Cambridge University Press, 2001.
\bibitem{MGM}
P. D. Mininni, D. O. G\'{o}mez and S. M. Mahajan, \textit{Dynamo action in magnetohydrodynamics and Hall magnetohydrodynamics}. The Astrophysics Journal 587 (2003), 472-481.

\bibitem{RWX}
X.X. Ren, J.H Wu, Z.Y Xiang and Z.F. Zhang, \textit{Global existence and decay of smooth solution for the 2-D
MHD equations without magnetic diffusion}. J. Funct. Anal. 267 (2014), 503-541.

\bibitem{SU}
D. A. Shalybkov and V. A. Urpin, \textit{The Hall effect and the decay of magnetic fields}. Astron. Astrophys. 321 (1997), 685-690.
\bibitem{ST}
 M. Sermange and R. Temam, \textit{Some mathematical questions related to the MHD equations}. Comm. Pure Appl. Math.
36 (1983), 635-664.
\bibitem{S}
E.M. Stein, \textit{Singular Integrals and Differentialbility Properties of Functions}. Princeton University Press, Princeton, 1970.
36 (1983), 635-664.
\bibitem{T}
J. B. Taylor, \textit{Relaxation of Toroidal Plasma and Generation of Reverse Magnetic Fields}. Phy. Rev. Letter 33 (1974), 1138-1141.
\bibitem{W}
M. Wardle, \textit{Star formation and the Hall effect}. Astrophysics and Space Science 292 (2004), 317-323.
\bibitem{YM}
K. Yamazaki and M. T. Moha, \textit{Well-posedness of Hall-magnetohydrodynamics system forced by L\'{e}vy noise}. Stoch. PDE: Anal. Comp. (2018), https://doi.org/10.1007/s40072-018-0129-6.
\bibitem{Zhang}
H.L. Zhang, \textit{The Well-posedness of Cauchy's problem for incompressible Hall magnetohydrodynamics in critical Besov spaces}. submitted to.
\bibitem{ZZ}
Y. Zhou and Y. Zhu, \textit{A class of large solutions to the 3D incompressible MHD and Euler equations with damping}. Acta Math. Sinica English Series 34 (2018), 63-78.	
\end{thebibliography}
\end{document}